\documentclass[12pt,a4paper,reqno]{amsart}
\usepackage[dvips]{graphicx}
\usepackage{amssymb}
\usepackage{latexsym}
\usepackage{amsmath}
\usepackage{bbm}
\usepackage{amsthm}
\newcommand{\R}{\mathbbm{R}}    
\providecommand{\ab}[1]{\vert #1\vert}		
\newcommand{\N}{\mathbbm{N}}    
\renewcommand{\leq}{\leqslant}
\renewcommand{\geq}{\geqslant}
\theoremstyle{plain}
\newtheorem{teo}{Theorem}

\newtheorem{lemma}{Lemma}
\newtheorem{prop}{Proposition}
\begin{document}
\title{The joints problem in $\R^n$}
\author{Ren\'e Quilodr\'an }\thanks{Research supported by NSF grant DMS-0401260}
\address{Department of Mathematics, University of California, Berkeley, CA 94720-3840, US}
\email{rquilodr@math.berkeley.edu}
\begin{abstract}
We show that given a collection of $A$ lines in $\R^n$, $n\geq 2$, the maximum number of their joints (points incident to at least $n$ lines whose directions form a linearly independent set) is $O(A^{n/(n-1)})$. An analogous result for smooth algebraic curves is also proven.
\end{abstract}
\maketitle
\section{Introduction}
In a recent paper, Katz and Guth \cite{KG} proved that the number of joints determined by a given collection of $A$ lines in $\R^3$ is $O(A^{3/2})$, where a joint (in $\R^3$) is a point which is incident to at least three non-coplanar lines of the given collection. Lately, Elekes, Kaplan and Sharir \cite{EKS} extended the results in \cite{KG} to obtain a bound on the number of incidences between a collection of lines and a given subset of their joints, in $\R^3$, which implies the result on the number of joints (they also consider a more general situation where joints are replaced by an arbitrary set of points satisfying that no plane contains more that $O(A)$ points and each point is incident to at least three lines). Both results make use of algebraic geometric properties of polynomials in three variables, which  bound the number of critical lines (lines where the polynomial and its gradient both vanish) a polynomial can have in terms of its degree. For more references in this problem consult \cite{KG} and \cite{EKS}.

Our proof does not require the algebraic geometric considerations in \cite{KG} and \cite{EKS} about polynomials in $n$ variables but just the fact that given $m$ points in $\R^n$, there exists a nonzero polynomial $Q\in\R[x_1,\dots,x_n]$ such that $Q$ vanishes on all the given $m$ points and whose degree is bounded by $d\lesssim m^{1/n}$. The method can be seen as, and was largely inspired by, an adaptation of the methods in \cite{Guth} to the discrete case, more precisely the result in the section ``warmup to multilinear Kakeya'' of \cite{Guth}, together with the application of the polynomial method as in \cite{Dvir}.

We point out that an independent proof of the bound on the number of joints, due to Kaplan, Sharir and Shustin \cite{KSS}, appeared at the same time as the one presented in the first version of this work. Our proof has some similarities with the proof in \cite{KSS} (for example, compare Lemma \ref{line-lemma} below and the ``Differentiating'' step in the proof of Theorem 1 in \cite{KSS}).

\section{The main result}

For a given collection of lines $L$ in $\R^n$ consider the set $J$ of points of the form $\cap_{i=1}^n \ell_i$, where $\ell_i\in L$ for all $1\leq i\leq n$ and the directions of the lines $\ell_1,\dots,\ell_n$, are linearly independent. We will refer to $J$ as the set of transverse intersections, or joints, of $L$.

{\bf Notation.} In this section the letters $L$ and $J$ will always be use with the same meaning, a set of lines in $\R^n$ and the set of joints determined by the set of lines, respectively. We will denote by $\ab{S}$ the cardinality of the set $S$. We also use the notation $X\lesssim Y$, $Y\gtrsim X$, $Y=\Omega(X)$ or $X=O(Y)$ to denote any estimate of the form $X\leq CY$ where $C$ is a constant that depends only on the dimension $n$. We use $X=\Theta(Z)$ to denote $X=O(Z)$ and $Z=O(X)$.

Our main Theorem is the following.

\begin{teo}
\label{main}
Let $L$ be a collection of lines in $\R^n$, then the cardinality of the set of joints of $L$, $J$, satisfies $\ab{J}\lesssim \ab{L}^{n/(n-1)}$.
\end{teo}

We start by proving the following Lemma.
\begin{lemma}
\label{line-lemma}
Let $J'$ be a subset of $J$ with the property that every line $\ell\in L$ with $\ell\cap J'\neq \emptyset$ contains at least $m$ points of $J'$, that is $\ab{\ell\cap J'}\geq m$, for some given constant $m$. Then $\ab{J'}\geq C_nm^n$, where $C_n$ is a constant depending on $n$ only.
\end{lemma}
\begin{proof}
By contradiction, assume there exists an arrangement of lines $L$ and points $J'$ as in the statement of the Theorem, where $\ab{J'}\leq \frac{m^n}{K}$, where $K$ is a big constant depending on $n$ only that we will choose later. Let $Q\in\R[x_1,\dots,x_n]$ be a nonzero polynomial that vanishes on every point of $J'$. We can choose $Q$ of degree $\text{deg}(Q)\leq c(n)\ab{J'}^{1/n}\leq \frac{c(n)}{K^{1/n}}m$ (because the space of polynomials of degree $\leq d$ has dimension $\binom{d+n}{d}=\Theta(d^n)$).
Choosing $K$ sufficiently big depending on $n$ only we can ensure that $\text{deg}(Q)<m$. The restriction of $Q$ to any line of $L$ which intersects $J'$ is a polynomial in one variable of degree $<m$ that vanishes on at least $m$ points, hence it vanishes identically. From $Q|_\ell=0$ we obtain $\nabla Q\cdot v|_\ell=0$, where $v$ is the direction of $\ell$. Therefore at each point of $J'$, $\nabla Q$ is orthogonal to a linearly independent set of $n$ vectors, so it is zero. Now every component of $\nabla Q$ vanishes on $J'$ and has degree $\text{deg}(\nabla Q)< \text{deg}(Q)<m$. We can apply the same argument to every component of $\nabla Q$, so inductively we obtain $\frac{\partial^{\alpha}Q}{\partial x^{\alpha}}=0$ on $J'$, for every multi-index $\alpha\in \N^n$. From here it follows that $Q$ is identically zero, which is a contradiction.

\end{proof}

Following the initial publication of this work, Fedor Nazarov observed that the proof of Theorem \ref{main} follows immediately from Lemma \ref{line-lemma}. We have left the original proof in the last section.

\begin{proof}[Proof of Theorem \ref{main}]
Let $m=K\ab{J}^{1/n}$, where $K$ satisfies $K^nC_n>1$ and $C_n$ is the constant in the conclusion of Lemma \ref{line-lemma} (hence $K$ depends on $n$ only). We start an iterative process to remove lines from $L$ having a control in the number of joints removed at each step. Let $L^{(0)}=L$ and $J^{(0)}=J$. Suppose that $L^{(i)}\subseteq L$, $L^{(i)}\neq \emptyset$ has been defined, and let $J^{(i)}\subseteq J$ denote the set of joints determined by $L^{(i)}$. With the choice of $m$, there must be a line $\ell_i\in L^{(i)}$ that contains no more than $m$ joints of $J^{(i)}$, otherwise, by Lemma \ref{line-lemma}, we would have $\ab{J}\geq\ab{J^{(i)}}\geq C_n m^n=K^nC_n\ab{J}>\ab{J}$ which is a contradiction. 

Define $L^{(i+1)}=L^{(i)}\backslash\{\ell_i\}$ and let $J^{(i+1)}$ be the set, possibly empty, of joints of $L^{(i+1)}$, which are necessarily contained in $J$. In this way we have $\ab{J^{(i)}}\leq \ab{J^{(i+1)}}+m$.

Since for $i\geq \ab{L}-(n-1)$ we have $J^{(i)}=\emptyset$, we conclude that $\ab{J}=\ab{J^{(0)}}\leq m\ab{L}=O(\ab{J}^{1/n}\ab{L})$, from where we obtain $\ab{J}\lesssim \ab{L}^{n/(n-1)}$.

\end{proof}

\section{The case of algebraic curves}
A similar bound as the one in Theorems \ref{main} can be proven if we replace lines by algebraic curves. By a smooth curve $\gamma$ we mean a curve such that its tangent vector $\dot\gamma$ exists at every point of $\gamma$ and is nonzero. Given a collection $\mathcal C$ of smooth curves we define the set of joints, $J$, determined by $\mathcal C$ as the set of incidences of at least $n$ curves in $\mathcal C$ such that the tangent vectors of the curves at the intersection are linearly independent. 

We start by considering a special case of algebraic curves. Let $\mathcal C$ be a set of smooth curves, each parametrized by polynomials, that is, if $\gamma\in \mathcal C$ we can parametrize it as $\gamma(t)=(P_1(t),\dots,P_n(t))$ where each $P_i$ is a polynomial in one variable of degree at most $d$, for a given constant $d$. We let $J$ denote the set of joints determined by $\mathcal C$.

A minor modification of Lemma \ref{line-lemma} gives the following.

\begin{lemma}
\label{line-lemma2}
Let $\mathcal C$ and $J$ be as in the previous paragraph, and let $J'$ be a subset of $J$ with the property that $\ab{\gamma\cap J'}\geq m$ for every curve $\gamma\in \mathcal C$ with $\gamma\cap J'\neq \emptyset$, for some given constant $m$. Then $\ab{J}=\Omega(m^n/d^n)$.
\end{lemma}

The conclusion follows as in the case of lines, and the bound on the number of joints is $\ab{J} \leq C
_n\ab{\mathcal C}^{n/(n-1)}d^{n/(n-1)}$, where $C_n$ is a constant depending on $n$ only.

More generally, if we consider an irreducible, smooth algebraic curve $\gamma$ of degree $d$ and if $Q\in \R[x_1,\dots,x_n]$ has degree $<m/d$ and its zero locus intersects $\gamma$ on at least $m$ different points, then the curve is contained in the zero set of $Q$, that is $Q|_\gamma\equiv0$, by an application of Bezout's Theorem (see for example Chapter 1 in \cite{Har} or Chapter 3 in \cite{Shaf}). Hence the same conclusion as in Lemma \ref{line-lemma2} holds if we let $\mathcal C$ consist of irreducible, smooth algebraic curves of degree at most $d$. Therefore we have the following Theorem.
\begin{teo}
Let $\mathcal C$ be a collection of irreducible, smooth algebraic curves of degree at most $d$ in $\R^n$. Let $J$ denote the set of joints determined by $\mathcal C$. Then the cardinality of $J$ satisfies $\ab{J}\leq C_n\ab{\mathcal C }^{n/(n-1)}d^{n/(n-1)}$, for some constant $C_n$ depending on $n$ only.
\end{teo}

\section{The original proof of Theorem \ref{main}}

We include here the original proof we had of Theorem \ref{main}. We derive the following consequence from Lemma \ref{line-lemma}. Let $c:J\to L$ be a function satisfying $x\in c(x)$ for all $x\in J$, that is, for each $x$, $c$ selects a line incident at $x$. Note that for each $x\in J$ we have at least $n$ transverse lines intersecting at $x$. Thus at each $x$ we have at least $n$ different lines to choose from. We call such a function a coloring of $J$.

\begin{prop}
\label{coloring}
There exists a coloring $c$ of $J$ such that for every line $\ell\in L$, \break$\ab{\{x\in \ell\cap J: c(x)=\ell\}}=O(\ab{J}^{1/n})$.
\end{prop}
\begin{proof}[Short proof, sketch]
We will use the same method as in the proof of Theorem \ref{main}. With the notation as in the proof of Theorem \ref{main} we know that for $i\geq \ab{L}-(n-1)$ we have $J^{(i)}=\emptyset$. Let $i_0\leq \ab{L}-(n-1)$ be the first time $J^{(i)}$ is empty. We let $\ell_i\in L$ be the line deleted at the $i-$th step, that is $\ell_i\in L^{(i)}\backslash L^{(i+1)}$. Every point in $J$ is contained in some line $\ell_i$, $1\leq i\leq i_0$. For $x\in J$ let $i(x)$ be the first time a line containing $x$ is deleted, ie, $x\in \ell_{i(x)}$ and $x\neq \ell_{i}$ for $i< i(x)$. Define $c(x)=\ell_{i(x)}$. Since $\ell_i$ is such that the number of joints of $L^{(i)}$ contained in $\ell_i$ is less than or equal to $m$, it follows that $\ab{\{x\in J: c(x)=\ell_i\}}\leq m$, and the proposition is verified.

\end{proof}

\begin{proof}[The original proof]
Let $m=\ab{J}$ and note that for any coloring $c$ of $J$, $W(\ell):=\ab{\{x\in \ell\cap J: c(x)=\ell\}}$ satisfies $W(\ell)\leq\ab{\ell\cap J}$. We use an inductive method to define the coloring $c$. Choose an ordering $J=\{x_1,\dots,x_m\}$. By a provisional coloring $c_\nu$ on $J_\nu:=\{x_1,\dots,x_\nu\}$ we mean a function $c_\nu:J_\nu\to L$ with $x\in c_\nu(x)$ for all $x\in J_\nu$. Given a provisional coloring $c_\nu$ we define the provisional counting function, $W_\nu$, on $L$ by $W_\nu(\ell)=\ab{\{x\in \ell\cap J_\nu: c_\nu(x)=\ell\}}$. We will say that the provisional coloring $c_\nu$ is acceptable if $W_\nu(\ell)\leq Km^{1/n}$ for all $\ell\in L$, for a given big constant $K$ depending only on $n$ that we will choose later. The Proposition is proven if we can find an acceptable coloring $c_m$.

Define the provisional coloring $c_\nu$ on $\{x_1,\dots,x_\nu\}$ inductively by setting $c_1(x_1)=\ell_1$, for an arbitrarily selected line $\ell_1\in L$ intersecting $x_1$. It follows that $W_1(\ell_1)=1$, $W_1(\ell)=0$, for all $\ell\neq \ell_1$, which is acceptable if we choose $K>1$. 

We will show that if $c_\nu$ is an acceptable coloring on $\{x_1,\dots,x_\nu\}$ then, by possibly modifying $c_\nu$, we can obtain an acceptable coloring $c_{\nu+1}$ on $\{x_1,\dots,x_{\nu+1}\}$.

Suppose $c_\nu$ is an acceptable coloring on $\{x_1,\dots,x_\nu\}$. The good case is the following: there is a line $\ell_{\nu+1}$ intersecting $x_{\nu+1}$ such that $W_{\nu}(\ell_{\nu+1})+1\leq Km^{1/n}$. In this case we let $c_{\nu+1}$ on $\{x_1,\dots,x_\nu,x_{\nu+1}\}$ be defined by $c_{\nu+1}(x_i)=c_\nu(x_i)$ for all $1\leq i\leq \nu$, and $c_{\nu+1}(x_{\nu+ 1})=\ell_{\nu+1}$ . It follows that $W_{\nu+1}(\ell)=W_{\nu}(\ell)$ for all $\ell\neq \ell_{\nu+1}$, and $W_{\nu+1}(\ell_{\nu+1})=W_\nu(\ell_{\nu+1})+1$, so that $c_{\nu+1}$ is acceptable. 

We now turn to the complementary case, the bad one. Here we have $W_\nu(\ell)\geq\frac{1}{2} Km^{1/n}$ (the $\frac{1}{2}$ is just because $Km^{1/n}$ may not be integer),  for all $\ell\in L$ incident at $x_{\nu+1}$, and we note that there are at least $n$ such lines with linearly independent directions. Now look at each point  $x_i\in\ell\cap J_\nu$ with $c_\nu(x_i)=\ell$, where $\ell$ is a line incident at $x_{\nu+1}$. If we can change the value of $c_\nu(x_i)$ to say $c_\nu(x_i)=\ell'$, for $\ell'\neq \ell$, for some $i$, without violating the restriction on $W_\nu(\ell')$ (that is $W_\nu(\ell')+1\leq Km^{1/n}$), then we are done, as we define $c_{\nu+1}(x_j)=c_\nu(x_j)$ for $j\leq \nu$ and $x_j\neq x_i$, $c_{\nu+1}(x_i)=\ell'$, $c_{\nu+1}(x_{\nu+1})=\ell$. If we can not find such $x_i$ this means that for any $\ell$ incident at $x_{\nu+1}$, for any point $x_i\in\ell\cap J_\nu$ with $c_\nu(x_i)=\ell$, and any $\ell'$ incident at $x_i$ we have $W_\nu(\ell')\geq \frac{1}{2}Km^{1/n}$. 

We let $I^{(1)}=\{x\in J: x\in\ell\cap J_\nu,\text{ for some } \ell\text{ incident at }x_{\nu+1}\text{ and } c_\nu(x)=\ell\}$ and if $I^{(\sigma)}$ is defined we let $I^{(\sigma+1)}=\{x\in J_\nu:\text{ there exists }x'\in I^{(\sigma)}\text{ and }\ell\text{ incident at }\break x'\text{ such that } x\in \ell\text{ and }c_\nu(x)=\ell\}$. We note that, similarly as we did for points in $I^{(1)}$, if $x\in I^{(\sigma)}$ and $\ell$ is such that $c_\nu(x)=\ell$ and there exists $\ell'\neq \ell$ incident at $x$ such that $W_\nu(\ell')+1\leq Km^{1/n}$, then by modifying $c_\nu$ on the corresponding points on $I^{(1)}\cup\dots\cup I^{(\sigma)}$, we can obtain an acceptable coloring $c_{\nu+1}$ on $\{x_1,\dots,x_{\nu+1}\}$ as desired.

If this is not the case, that means that for any $x\in\bigcup\limits_{n=1}^{\infty}I^{(\sigma)}=:J'$ and any $\ell\in L$ of the at least $n$ transverse lines incident at $x$ we have $W_\nu(\ell)\geq\frac{1}{2}Km^{1/n}$. Note that $I^{(\sigma+1)}=I^{(\sigma)}$ for all sufficiently large $\sigma$, since these are nested subsets of the finite set $J$. We let $L'$ denote the set of lines of $L$ incident to some point of $J'$. Thus for all $\ell\in L'$ we have
\begin{equation}
 \label{counting-estimate}
\frac{1}{2}Km^{1/n}\leq W_\nu(\ell)\leq\ab{\ell\cap J'},
\end{equation}
where the second inequality comes from the inclusion $\{x\in \ell\cap J_\nu:c_\nu(x)=\ell\}\subseteq \ell\cap J'$, that we show now. We first note that if $x\in I^{(\sigma)}$ is such that $c_\nu(x)=\ell$, then any $x'\in \ell\cap J_\nu$ with $c_\nu(x')=\ell$ is in $I^{(\sigma)}$. Now for $\ell\in L'$ we have $\ell\cap J'\neq\emptyset$, so let $x_{i_0}\in \ell\cap J'$. For $x_{i_0}$ we have, $x_{i_0}\in J'$ hence $x_{i_0}\in I^{(\sigma)}$ for some $\sigma\geq 1$, then any $x\in\ell\cap J_\nu$ with $c_\nu(x)=\ell$ is in either in $I^{(\sigma)}$ or in $I^{(\sigma+1)}$ (depending whether $c_\nu(x_{i_0})=\ell$ or not), thus $x\in J'$ and the inclusion follows.

Now use Lemma \ref{line-lemma} together with (\ref{counting-estimate}) applied to $L'$ and $J'$, to obtain \break$\ab{J'}\geq C(n)(\frac{1}{2}Km^{1/n})^n=\frac{1}{2^n}C(n)K^{n}m$. We now choose $K$ big enough, depending on $n$ only so that $\frac{1}{2^n}C(n)K^{n}>1$. Hence we obtain $\ab{J}\geq \ab{J'}>m=\ab{J}$ which is a contradiction. This means that in the bad case we can always modify $c_\nu$ to obtain an acceptable coloring $c_{\nu+1}$. Therefore the Proposition is proved, by induction.

\end{proof}

For those familiar with \cite{Guth}, a coloring as in Proposition \ref{coloring} is the analog in ``warmup to multilinear Kakeya'' in \cite{Guth} to finding directions $v_{j(k),a(k)}$ such that for the $k$-th cube $Q_k$, the directed volume $V_{Z\cap Q_k}(v_{j(k),a(k)})$ is large ($V_{Z\cap Q_k}(v_{j(k),a(k)})\gtrsim 1$).
The next proposition follows exactly as in the last paragraphs in the mentioned section of \cite{Guth}.

\begin{proof}[Proof of Theorem \ref{main}]
By Proposition \ref{coloring} there exists a coloring $c$ satisfying $\ab{\{x\in \ell\cap J: c(x)=\ell\}}=O(\ab{J}^{1/n})$ for all $\ell\in L$. For each $x\in J$ we have a distinguished line, namely $c(x)$. We have just associated a line to any point $x\in J$. There are in total $\ab{L}$ lines and $\ab{J}$ points. By the pigeonhole principle, there is a line, $\ell^*$, associated to $\gtrsim \ab{J}/\ab{L}$ different points, therefore $\ab{\{x\in \ell^*\cap J: c(x)=\ell^*\}}\gtrsim \ab{J}/\ab{L}$.

On the other hand  $\ab{\{x\in \ell^*\cap J: c(x)=\ell^*\}}\lesssim \ab{J}^{1/n}$. From here it follows that $\ab{J} \lesssim \ab{L}^{n/(n-1)}$.
\end{proof}

\vspace{1cm}
{\bf Acknowledgements.} I am grateful to my dissertation advisor, Michael Christ, for many helpful comments. I thank Fedor Nazarov for pointing out the simplification of the proof of Theorem \ref{main}.


\end{document}